\documentclass{amsart}
\usepackage{amsmath,amsthm}
\usepackage{amsfonts,amssymb}
\usepackage{amsfonts}
\usepackage[all]{xy}
\usepackage{pdfsync}
\usepackage{hyperref}
\usepackage{graphicx}
\usepackage{subfigure}
\usepackage{tikz}%
\usepackage{color}
\usepackage{booktabs}
 \usepackage{multirow}
\usepackage{caption}
\usepackage{mathrsfs}
\allowdisplaybreaks[4]

\newtheorem{theorem}{Theorem}[section]
\newtheorem{lemma}[theorem]{Lemma}
\newtheorem{corollary}[theorem]{Corollary}
\theoremstyle{definition}
\newtheorem{definition}[theorem]{Definition}
\newtheorem{ex}[theorem]{Example}

\newtheorem{proposition}[theorem]{Proposition}
\theoremstyle{remark}
\newtheorem{remark}[theorem]{Remark}
\numberwithin{equation}{section}

\begin{document}

\title[Dunkl approach to slice regular functions]{Dunkl approach to  slice regular functions}

\dedicatory{Dedicated to the memory of Frank Sommen}
%

\author[G. Binosi]{Giulio Binosi}
\address[1]{Dipartimento di Matematica e Informatica ``U. Dini" \\Universita degli studi di Firenze\\Viale Morgagni 67/A, I-50134 Firenze\\Italy.}
\email{binosi@altamatematica.it}
\thanks{}

\author[H. De Bie]{Hendrik De Bie}
\address[2]{Clifford research group \\Department of Electronics and Information Systems \\Faculty of Engineering and Architecture\\Ghent University\\Krijgslaan 281, 9000 Ghent\\ Belgium.}
\curraddr{}
\email{Hendrik.DeBie@ugent.be}
\thanks{}

\author[P. Lian]{Pan Lian}
\address[3]{School of Mathematical Sciences\\ Tianjin Normal University\\
Binshui West Road 393, 300387 Tianjin\\ P.R. China.
}
\curraddr{}
\email{panlian@tjnu.edu.cn}
\thanks{}


\subjclass[2020]{Primary 30G35; Secondary  33C52, 33C55}
\keywords{Fueter's theorem, Dirac operator, Dunkl operator, monogenic functions, slice regular functions.}
\date{}


\begin{abstract}In this paper, \textcolor{black}{we establish a connection between Dunkl analysis and slice analysis in the setting of Clifford algebras. Specifically,} we show that a Clifford algebra-valued function is slice if, and only if, it belongs to the kernel of the Dunkl-spherical Dirac operator and that a slice function is slice regular if, and only if, it lies in the kernel of the Dunkl-Cauchy-Riemann operator for a suitable parameter. Based on this correspondence and the inverse Dunkl intertwining operator, we propose a new method to construct a family of classical monogenic functions from a given holomorphic function, in the spirit of Fueter’s theorem.
\end{abstract}

\maketitle
\section{Introduction}
In the literature, there are two main approaches to generalize classical holomorphic function theory to higher dimensions. The first is the function theory of several complex variables (see e.g.\,\cite{kra}), while  the second revolves around the function theory of the Dirac operator,  usually referred to as Clifford analysis \cite{dss}. Functions in the kernel of the Dirac operator are called monogenic.  Over the past four decades, many results from complex analysis of one variable have been  generalized within the latter framework. 
Unfortunately, the powers of the variables are not in the kernel of the classical Dirac operator.  Since the 1990s,  Leutwiler, Eriksson, and their collaborators have developed a modification of classical Clifford analysis that incorporates the powers of the variable into the kernel of the modified Dirac operator, see e.g.\,\cite{leu1, leu2}.  Interestingly, their adjustment is closely related to  hyperbolic geometry, see e.g.\,\cite{eo}. Around 2005, another  generalization of classical Clifford analysis was introduced in \cite{ckr} by employing  the influential Dunkl operators \cite{dun}, which has attracted significant attention in recent years. Some relations between these three types of monogenic functions have been discussed in e.g.\,\cite{eo, gos, leu2}. See the right half of Fig. \ref{fig11} for a rough understanding.

In 2006-2007, a new hypercomplex function theory, mainly studying so-called slice regular functions, was introduced  by Gentili and Struppa in \cite{gs}. This theory has attracted considerable interest and experienced  a rapid growth over the last years. In fact,  the theory of  slice regular functions  of one variable nowadays is well-established  and has found significant  applications in quaternionic quantum mechanics and spectral theory of several operators, see e.g.\,\cite{css1}.  Importantly, a convergent power series of the form $\sum_{n=0}^{\infty}x^{n}a_{n}$ defines a slice regular function on an open ball centered at the origin.  There is a vast amount of publications on this topic. Without claiming completeness, we refer to \cite{css22, css, gs, gss}.

Fueter's theorem  (see e.g.\,\cite{css2, esv, fue, qian}) establishes a link between slice regular (or slice monogenic) functions and monogenic functions.  These two theories are also connected through the Radon transform \cite{clss}. Besides those links, the theories of hypermonogenic functions, Dunkl-Clifford analysis and slice regular functions have all been developed  independently to some extent.

In this paper, we establish a link between these theories providing a new characterization of slice and  slice regular functions using the Dunkl-spherical Dirac and the Dunkl-Cauchy-Riemann operators. More precisely, we show that a function defined in an axially symmetric open set is slice if and only if it lies in the kernel of the Dunkl-spherical Dirac operator (Proposition \ref{prop charcaterization sliceness}) and, for slice domains, it is slice regular if it is in the kernel of both the Dunkl-spherical Dirac operator and the Dunkl-Cauchy-Riemann operator (Theorem \ref{sm1}), under an appropriate condition on the multiplicity function.
By this observation, we introduce a new approach in Section \ref{fna} to construct a family of classical monogenic functions from a given slice regular function, based on the inverse Dunkl intertwining operator \cite{dx1}. In contrast,  Fueter's famous theorem only produces a single axial monogenic function at a time. Moreover, our approach  maps each $x^{n}$ homogeneously to monogenic polynomials of the same degree and does not  exhibit significant differences between even and odd dimensions.  The resulting formula bears a resemblance to the well-known Cauchy-Kovalevskaya extension \cite{dss}. We also remark that the structure of the inverse Dunkl intertwining operator is much simpler than the Dunkl intertwining operator, as the latter one is explicitly known only in some special cases, see e.g.\,\cite{dl, dx1}. The connection with hyperbolic harmonic functions provides a local approach to study  slice regular functions, as shown in Theorem \ref{sld1}. Moreover, we find that the functions produced via Fueter's scheme also lie in the kernel of certain Dunkl-Cauchy-Riemann operators. The Laplace operator in this scheme acts as a shift operator on the index $\gamma_{\kappa}$, see Section  \ref{fuet1}. These  relations are summarized in Figure \ref{fig11}.

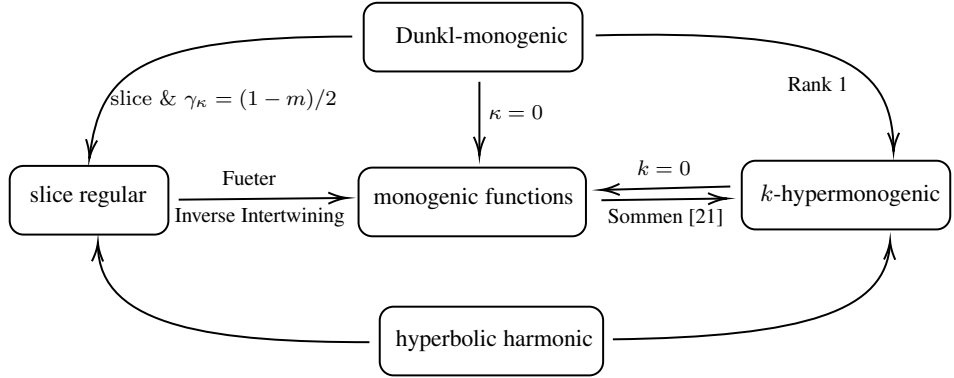
\begin{figure}
\tikzset{every picture/.style={line width=0.75pt}} 

\begin{tikzpicture}[x=0.65pt,y=0.65pt,yscale=-1,xscale=1]

\draw   (224,52) .. controls (224,47.58) and (227.58,44) .. (232,44) -- (343.33,44) .. controls (347.75,44) and (351.33,47.58) .. (351.33,52) -- (351.33,76) .. controls (351.33,80.42) and (347.75,84) .. (343.33,84) -- (232,84) .. controls (227.58,84) and (224,80.42) .. (224,76) -- cycle ;
\draw    (291.33,90) -- (291.33,131) ;
\draw [shift={(291.33,133)}, rotate = 270] [color={rgb, 255:red, 0; green, 0; blue, 0 }  ][line width=0.75]    (10.93,-3.29) .. controls (6.95,-1.4) and (3.31,-0.3) .. (0,0) .. controls (3.31,0.3) and (6.95,1.4) .. (10.93,3.29)   ;
\draw   (221,148) .. controls (221,143.58) and (224.58,140) .. (229,140) -- (345.33,140) .. controls (349.75,140) and (353.33,143.58) .. (353.33,148) -- (353.33,172) .. controls (353.33,176.42) and (349.75,180) .. (345.33,180) -- (229,180) .. controls (224.58,180) and (221,176.42) .. (221,172) -- cycle ;
\draw   (18,147) .. controls (18,142.58) and (21.58,139) .. (26,139) -- (102.33,139) .. controls (106.75,139) and (110.33,142.58) .. (110.33,147) -- (110.33,171) .. controls (110.33,175.42) and (106.75,179) .. (102.33,179) -- (26,179) .. controls (21.58,179) and (18,175.42) .. (18,171) -- cycle ;
\draw    (116.33,158.67) -- (212.33,157.69) ;
\draw [shift={(214.33,157.67)}, rotate = 179.42] [color={rgb, 255:red, 0; green, 0; blue, 0 }  ][line width=0.75]    (10.93,-3.29) .. controls (6.95,-1.4) and (3.31,-0.3) .. (0,0) .. controls (3.31,0.3) and (6.95,1.4) .. (10.93,3.29)   ;
\draw   (444,144) .. controls (444,139.58) and (447.58,136) .. (452,136) -- (557.33,136) .. controls (561.75,136) and (565.33,139.58) .. (565.33,144) -- (565.33,168) .. controls (565.33,172.42) and (561.75,176) .. (557.33,176) -- (452,176) .. controls (447.58,176) and (444,172.42) .. (444,168) -- cycle ;
\draw    (362.33,159) -- (433.33,158.03) ;
\draw [shift={(435.33,158)}, rotate = 179.22] [color={rgb, 255:red, 0; green, 0; blue, 0 }  ][line width=0.75]    (10.93,-3.29) .. controls (6.95,-1.4) and (3.31,-0.3) .. (0,0) .. controls (3.31,0.3) and (6.95,1.4) .. (10.93,3.29)   ;
\draw    (437.33,151) -- (364,152.95) ;
\draw [shift={(362,153)}, rotate = 358.48] [color={rgb, 255:red, 0; green, 0; blue, 0 }  ][line width=0.75]    (10.93,-3.29) .. controls (6.95,-1.4) and (3.31,-0.3) .. (0,0) .. controls (3.31,0.3) and (6.95,1.4) .. (10.93,3.29)   ;
\draw   (233.33,229) .. controls (233.33,224.58) and (236.92,221) .. (241.33,221) -- (355.33,221) .. controls (359.75,221) and (363.33,224.58) .. (363.33,229) -- (363.33,253) .. controls (363.33,257.42) and (359.75,261) .. (355.33,261) -- (241.33,261) .. controls (236.92,261) and (233.33,257.42) .. (233.33,253) -- cycle ;
\draw    (227.33,241.67) .. controls (139.22,241.67) and (69.74,233.83) .. (69.32,185.15) ;
\draw [shift={(69.33,183.67)}, rotate = 91.15] [color={rgb, 255:red, 0; green, 0; blue, 0 }  ][line width=0.75]    (10.93,-3.29) .. controls (6.95,-1.4) and (3.31,-0.3) .. (0,0) .. controls (3.31,0.3) and (6.95,1.4) .. (10.93,3.29)   ;
\draw    (369,240) .. controls (482.6,238.69) and (527.97,223.14) .. (530.26,184.45) ;
\draw [shift={(530.33,182.67)}, rotate = 91.43] [color={rgb, 255:red, 0; green, 0; blue, 0 }  ][line width=0.75]    (10.93,-3.29) .. controls (6.95,-1.4) and (3.31,-0.3) .. (0,0) .. controls (3.31,0.3) and (6.95,1.4) .. (10.93,3.29)   ;
\draw    (217.33,63.67) .. controls (91.6,63.67) and (70.74,85.23) .. (64.52,133.2) ;
\draw [shift={(64.33,134.67)}, rotate = 276.98] [color={rgb, 255:red, 0; green, 0; blue, 0 }  ][line width=0.75]    (10.93,-3.29) .. controls (6.95,-1.4) and (3.31,-0.3) .. (0,0) .. controls (3.31,0.3) and (6.95,1.4) .. (10.93,3.29)   ;
\draw    (359,63) .. controls (494.96,59.7) and (531.6,76.33) .. (532.32,129.06) ;
\draw [shift={(532.33,130.67)}, rotate = 270] [color={rgb, 255:red, 0; green, 0; blue, 0 }  ][line width=0.75]    (10.93,-3.29) .. controls (6.95,-1.4) and (3.31,-0.3) .. (0,0) .. controls (3.31,0.3) and (6.95,1.4) .. (10.93,3.29)   ;

\draw (241,56) node [anchor=north west][inner sep=0.75pt]   [align=left] {{\fontfamily{ptm}\selectfont {\small Dunkl-monogenic}}};
\draw (228,150) node [anchor=north west][inner sep=0.75pt]   [align=left] {{\fontfamily{ptm}\selectfont {\small monogenic functions}}};
\draw (28,149) node [anchor=north west][inner sep=0.75pt]   [align=left] {{\fontfamily{ptm}\selectfont {\small slice regular  \ \ }}};
\draw (140,140) node [anchor=north west][inner sep=0.75pt]   [align=left] {{\footnotesize {\fontfamily{ptm}\selectfont Fueter}}};
\draw (453,148) node [anchor=north west][inner sep=0.75pt]   [align=left] {{\small $\displaystyle k${\fontfamily{ptm}\selectfont -hypermonogenic}}};
\draw (364.33,162) node [anchor=north west][inner sep=0.75pt]   [align=left] {{\fontfamily{ptm}\selectfont {\footnotesize Sommen \cite{gos}}}};
\draw (75,92.4) node [anchor=north west][inner sep=0.75pt]  [font=\footnotesize]  {slice \& ${\textstyle \gamma_{\kappa} =(1-m)/2}$};
\draw (295,101.9) node [anchor=north west][inner sep=0.75pt]  [font=\footnotesize]  {${\textstyle \kappa =0}$};
\draw (381.33,133.9) node [anchor=north west][inner sep=0.75pt]  [font=\footnotesize]  {${\textstyle k=0}$};
\draw (469,85) node [anchor=north west][inner sep=0.75pt]   [align=left] {{\footnotesize {\fontfamily{ptm}\selectfont Rank 1}}};
\draw (112.33,161.67) node [anchor=north west][inner sep=0.75pt]   [align=left] {{\footnotesize {\fontfamily{ptm}\selectfont Inverse Intertwining}} };
\draw (241,233) node [anchor=north west][inner sep=0.75pt]   [align=left] {{\small {\fontfamily{ptm}\selectfont hyperbolic harmonic}}};
\end{tikzpicture}
 \centering
        \caption{Dunkl-monogenic, slice regular and hypermonogenic functions on $\Omega_D$ centered at the origin }
        \label{fig11}
\end{figure}
The rest of this paper is structured  as follows. In Section \ref{pre}, we recall some basic notions of Dunkl-Clifford analysis and slice regular functions. Section \ref{dsu1} is dedicated  to the characterization of slice regular functions through Dunkl-Cauchy-Riemann operator. In Section \ref{fuet1}, we revisit the classical Fueter's theorem and highlight its correlation with Dunkl-Clifford analysis. In Section \ref{fna}, we present our new approach for constructing monogenic functions from a given holomorphic function. Some examples are given as well.


\section{Preliminaries}\label{pre}
\subsection{Dunkl operators}
Suppose $\alpha$ is a non-zero vector in $\mathbb{R}^{m+1}$. The reflection in the hyperplane orthogonal to $\alpha$ is defined by
\begin{equation*}
    \sigma_{\alpha} x:=x-2\frac{\langle \alpha, x\rangle}{|\alpha|^{2}}\alpha, \qquad  \mbox{for} \quad  x\in \mathbb{R}^{m+1},
\end{equation*}
where $\langle \alpha, x\rangle:=\sum_{j=0}^{m}\alpha_{j}x_{j}$ is the standard Euclidean inner product and $|x|:=\sqrt{\langle x, x\rangle}$.
A finite set $\mathscr{R}\subset \mathbb{R}^{m+1}\backslash\{0\}$ is called a root system if $\mathscr{R}\cap \mathbb{R}\cdot \alpha=\{\alpha, -\alpha\}$
and $\sigma_{\alpha} \mathscr{R}=\mathscr{R}$ for all $\alpha\in \mathscr{R}$. \textcolor{black}{ The dimension of ${\rm span}_{\mathbb{R}}\mathscr{R}$ is called the rank of $\mathscr{R}$.} For a given root system $\mathscr{R}$, the reflections $\sigma_{\alpha}$, $\alpha\in \mathscr{R}$, generate a finite group $W\subset {\rm O}(m+1)$, called the Coxeter-group associated with $\mathscr{R}$. It is known that each root system can be expressed  as a disjoint union, i.e. $\mathscr{R}=\mathscr{R}_{+}\cup (-\mathscr{R}_{+})$, where $\mathscr{R}_{+}$ and $-\mathscr{R}_{+}$ are separated by a hyperplane through the origin. We fix the positive subsystem $\mathscr{R}_{+}$. \textcolor{black}{From now on, we assume that $\mathscr{R}$ is normalized, meaning that $\langle \alpha, \alpha\rangle=2$ for every $\alpha\in \mathscr{R}$.} A function $\kappa: \mathscr{R}\rightarrow \mathbb{R}$ on a root system $\mathscr{R}$ is called a multiplicity function if it is invariant under the action of the associated reflection group $W$. For later use, we introduce the index $\gamma_{\kappa}:=\sum_{\alpha\in \mathscr{R}_{+}}\kappa(\alpha).$ \textcolor{black}{ The number $(m+1)+2\gamma_{\kappa}$ is  the generalized homogeneous dimension, see \cite{dx1}.}
\begin{definition} \cite{dun}
Let $\mathscr{R}$ be a root system over $\mathbb{R}^{m+1}$. For each fixed positive subsystem $\mathscr{R}_{+}$ and multiplicity function $\kappa$, the  Dunkl operator is defined by
\begin{equation*}
    T_{i}f(x)=\frac{\partial}{\partial x_{i}} f(x)+\sum_{\alpha\in \mathscr{R}_{+}}\kappa(\alpha)\frac{f(x)-f(\sigma_{\alpha }x)}{\langle \alpha, x\rangle}\alpha_{i}, \qquad i=0,1,\ldots, m,
\end{equation*}
for $f\in \mathcal{C}^{1}(\mathbb{R}^{m+1})$.
\end{definition}
\begin{remark} The Dunkl operators $\{T_{j}\}_{j=0}^{m}$ reduce to the classical partial derivatives when $\kappa\equiv 0$.  An important property of Dunkl operators is that for fixed $\kappa$, they commute mutually, i.e., $T_{i}T_{j}=T_{j}T_{i}$.
\end{remark}

 The classical Fischer product can be defined in the Dunkl setting. Denote $\mathcal{P}$ for the space of polynomial functions on $\mathbb{R}^{m+1}$.
\begin{definition} \cite{ros} For any $p, q\in \mathcal{P}$, we define
\begin{equation} \label{bil2}
    [p, q]_{\kappa}:=(p(T)q)(0).
\end{equation}
  The set of all values $\kappa$ such that the bilinear form  \eqref{bil2} is degenerate is called  the singular set of $\kappa$, denoted by $K^{\rm sing}$.
\end{definition}

 A complete description of $K^{\rm sing}$ was given in \cite{ddo}.
 \begin{ex}
     For the Dunkl operator on the line, the singular set is given by
 \begin{equation*}
     K^{\rm sing}=\{-1/2-n, n\in \mathbb{Z}_{+}\}.
 \end{equation*}
 \end{ex}
 For more details on Dunkl theory, we refer to  \cite{dx1, ros}.  In this work, we are mainly interested in the whole sum $\gamma_{\kappa}=\sum_{\alpha\in \mathscr{R}_{+}}\kappa(\alpha)$, rather than the  concrete  value of the multiplicity function $\kappa$. Therefore,   in the subsequent sections where we  discuss the index $\gamma_{\kappa}$, the singular set of $\kappa$ will   be excluded for simplicity.
\subsection{Dunkl-Clifford analysis}
We denote by $\mathbb{R}_{0,m}$ the real universal Clifford algebra  generated by the multiplication rules
\begin{equation*}
    		\begin{split}
    			\left\{
    			\begin{array}{ll}
    				e_{j}\,e_{k}=-e_{k}\,e_{j}, & 1\le j\neq k\le m,\\
    				e_{j}^{2}=-1, & j=1,\ldots, m,\\
    			\end{array}
    			\right.
    		\end{split}
    	\end{equation*}
where $(e_{1}, e_{2}, \ldots, e_{m})$ is an orthonormal basis of the Euclidean space $\mathbb{R}^{m}$. In the following, any vector  $ (x_{1}, x_{2}, \ldots, x_{m})\in \mathbb{R}^{m}$ will be identified with
$\underline{x}=\sum_{j=1}^{m}x_{j}e_{j}$. \textcolor{black}{The Clifford product of two vectors splits into a scalar part  and a bivector as follows,
\begin{equation} \label{iw1}
    \underline{x} \,\underline{y}=-\langle \underline{x}, \underline{y}\rangle+\underline{x}\wedge \underline{y},
\end{equation}
where \begin{equation*}
    \langle \underline{x}, \underline{y}\rangle=\sum_{j=1}^{m}x_{j}y_{j} \quad
{\rm and} \quad
    \underline{x}\wedge \underline{y}=\sum_{j=1}^{m}\sum_{k=j+1}^{m}e_{j}e_{k}(x_{j}y_{k}-x_{k}y_{j}).
\end{equation*}}
The space of  paravectors  is defined  by
\begin{equation*}
    \{x=x_{0}+e_{1}x_{1}+\cdots+e_{m}x_{m}|x_{0}, \ldots, x_{m}\in \mathbb{R}\}\simeq \mathbb{R}^{m+1}.
\end{equation*}
 A basis for $\mathbb{R}_{0,m}$ is given by the elements $e_{A}=e_{i_{1}}\ldots e_{i_{j}}$, where $A=\{i_{1}, \ldots, i_{j}\}\subset \{1, \ldots, m\}$ satisfying $1\le i_{1}<\cdots<i_{j}\le m$. Thus an $\mathbb{R}_{0,m}$-valued function $f$ over $\Omega \subset \mathbb{R}^{m+1} $ can be written as
$
    f=\sum_{A}e_{A}f_{A},
$
where each component $f_{A}: \Omega\rightarrow \mathbb{R}$ is a real valued function. If each component $f_{A}\in \mathcal{C}^{1}(\Omega)$, we say that  $f$ is of class  $\mathcal{C}^{1}(\Omega)$.  We also need the conjugate (the anti-involution) on $\mathbb{R}_{0,m}$, defined by $e_{0}^c=e_{0}$, $e_{j}^c=-e_{j}$ and $(uv)^c=v^c\,u^c$ for any $u, v\in \mathbb{R}_{0,m}$.

In this  paper, we will always assume that the finite reflection group $W$ associated to the root system $\mathscr{R}$ leaves the $x_{0}$-axis invariant and assume that $T_{0}=\partial_{x_{0}}$. Namely, we demand that the root system $\mathscr{R}\subset \mathbb{R}^{m}$.
\begin{definition}
  The Dunkl-Dirac operator on $\mathbb{R}^{m}$ associated to the root system $\mathscr{R}$ is defined  by
\begin{equation*}
    \underline{D}_{\kappa}:=\sum_{j=1}^{m} e_{j}T_{j}.
\end{equation*}
Furthermore, the Dunkl-Cauchy-Riemann operator on $\mathbb{R}^{m+1}$  is given by
\begin{equation*}
    D_{\kappa}:=\partial_{x_0}+\underline{D}_{\kappa}.
\end{equation*}
Functions in the kernel of $D_\kappa$ or $\underline D_\kappa$ will be called Dunkl monogenic. \end{definition}
\begin{remark} When $\kappa\equiv 0$,   $\underline{D}_{\kappa}$  reduces to the Euclidean  Dirac operator on $\mathbb{R}^{m}$, i.e.
\begin{equation*}
  \underline{D}=\sum_{j=1}^{m} e_{j}\partial_{x_j},
\end{equation*}
while $D_{\kappa}$ reduces to the generalized Cauchy-Riemann operator $\overline{\partial}_{x}:=\partial_{x_0}+\underline{D}.$
\end{remark}

We consider the spherical part of the Dunkl-Dirac operator.
\begin{lemma} \label{fc1}\cite{fck} The Dunkl-Dirac operator on $\mathbb{R}^{m}$ in spherical coordinates is given by
\begin{equation*}
    \underline{D}_{\kappa}=\underline{\omega}\left(\partial_{r}+\frac{1}{r}\widetilde{\Gamma}_{\underline{\omega}}\right),
\end{equation*}
where $\underline{\omega}=\underline{x}/|\underline{x}|$, $r=|\underline{x}|$ and
\begin{equation}\label{ga1}
    \widetilde{\Gamma}_{\underline{\omega}}:=\gamma_{\kappa}+\Gamma_{\underline{\omega}}+\Psi
\end{equation}
is the Dunkl-spherical Dirac operator, in which 
\begin{equation*}
    \Gamma_{\underline{\omega}}=-\sum_{i<j}e_{i}e_{j}(x_{i}\partial_{x_{j}}-x_{j}\partial_{x_{i}})
\end{equation*} is the  spherical Dirac operator and
\begin{equation}\label{p21}
\begin{split}
     \Psi f(\underline{x})&:=-\sum_{i<j}e_{i}e_{j}\sum_{\underline{\alpha}\in \mathscr{R}_{+}}\kappa(\underline{\alpha})\frac{f(\underline{x})-f(\sigma_{\underline{\alpha}}\underline{x})}{\langle \underline{\alpha}, \underline{x}\rangle}(x_{i}\alpha_{j}-x_{j}\alpha_{i})\\
     &\quad -\sum_{\underline{\alpha}\in \mathscr{R}_{+}}\kappa(\underline{\alpha})f(\sigma_{\underline{\alpha}} \underline{x})
\end{split}
\end{equation}
for any $f\in \mathcal{C}^{1}(\mathbb{R}^{m})$.
\end{lemma}

In the Dunkl case, we have the following generalization of a property of the ordinary spherical Dirac operator.
\begin{lemma} \cite{fck} Suppose $M_{k}(\underline{x})$ is a homogeneous Dunkl monogenic polynomial of degree $k$ on $\mathbb{R}^{m}$, i.e., $M_{k}\in\ker\underline D_\kappa$ and $M_{k}(\underline x)=r^kM_k(\underline\omega)$. Then it holds
\begin{equation}\label{pk1}
\begin{split}
    			\left\{
    			\begin{array}{ll}
    				 \widetilde{\Gamma}_{\underline\omega}(M_j)=-jM_j,\\
    				\widetilde{\Gamma}_{\underline\omega}(\underline xM_j)=(2\gamma_{\kappa}+j+m-1)\underline xM_j.\\
    			\end{array}
    			\right.
    		\end{split}
\end{equation}
Furthermore, for any radial function $f(\underline{x})=f(|\underline{x}|)=f(r)$ and $\underline{\omega}=\underline{x}/|\underline{x}|,$ we have
\begin{equation}\label{pk12}
\begin{split}
    			\left\{
    			\begin{array}{ll}
    				 \widetilde{\Gamma}_{\underline\omega}f(r)=0,\\
    				\widetilde{\Gamma}_{\underline\omega}(\underline{\omega}f(r))=\widetilde{\Gamma}_{\underline\omega}(\underline{\omega})f(r),\\
          \widetilde{\Gamma}_{\underline\omega}(\underline{\omega})=(2\gamma_{\kappa}+m-1)\underline{\omega}.
    			\end{array}
    			\right.
    		\end{split}
\end{equation}
\end{lemma}
~\\

\subsection{Slice regular functions}
Let $D\subset\mathbb{C}$ be an open subset symmetric with respect to the real axis. A function $F\colon D\to \mathbb{R}_{0, m}\otimes\mathbb{C}$ is called stem function if it is complex intrinsic, i.e., it satisfies
\begin{equation*}
    F(\overline{z})=\overline{F(z)},\qquad\forall \, z\in D,
\end{equation*}
where $\overline{z}$ is the complex conjugation of $z$, and $\overline{(x\otimes z)}:= x\otimes \overline{z}$, for any $x\in\mathbb{R}_{0, m}$, $z\in\mathbb{C}$. Note that, if $F=\alpha+i\beta$, it is equivalent to requiring $\alpha(\overline{z})=\alpha(z)$ and $\beta(\overline{z})=-\beta(z)$.
Let $\mathbb{S}^{m-1}$ be the $(m-1)$-dimensional unit sphere in $\mathbb{R}^{m}$. Then for any  vector $ \underline{\omega} \in \mathbb{S}^{m-1}$, it is easy to verify that $\underline{\omega}^{2}=-1$. Given $\underline\omega\in\mathbb{S}^{m-1}$, let $\mathbb{C}_{\underline\omega}={\rm span}\{1,\underline\omega\}$ be the complex slice of $\mathbb{R}^{m+1}$ generated by $\underline\omega$ and define the $^\ast$-algebra isomorphism $\phi_{\underline\omega}:\mathbb{C}\ni a+ib\mapsto a+\underline\omega b\in\mathbb{C}_{\underline\omega}$. We can define the circularization of a symmetric set $D$ as
\begin{equation*}
    \Omega_D:=\bigcup_{\underline\omega\in\mathbb{S}^{m-1}}\phi_{\underline\omega}(D)=\{a+\underline\omega b: a+ib\in D, \underline\omega\in\mathbb{S}^{m-1}\}\subset\mathbb{R}^{m+1}.
\end{equation*}
Sets of the form $\Omega_D$ are called axially symmetric and slice domains whenever $D\cap\mathbb{R}\neq\emptyset$.
Given a stem function $F=\alpha+i \beta$, we define the induced slice function $f:=\mathcal{I}(F):\Omega_D\to\mathbb{R}_m$ for any $x=x_0+\underline\omega|\underline x|\in\Omega_D$ as
\begin{equation}\label{sf1}
    f(x):=\alpha(x_0,|\underline x|)+\underline{\omega} \, \beta(x_0,|\underline x|).
\end{equation}
We will say that $f$ is slice regular if it is induced by a stem function which is holomorphic with respect to the complex structure on $\mathbb{R}_{0, m}\otimes\mathbb{C}$ defined by left multiplication by $i$. Equivalently, a slice function $f$ is slice regular on $\Omega_D$ if and only if for  any unit vector $\underline{\omega}\in \mathbb{S}^{m-1}$,  the restriction $f_{\underline{\omega}}$  of the function $f$ to the complex plane $\mathbb{C}_{\underline{\omega}}$ is holomorphic, i.e.
\begin{equation*}
    \overline{\partial}_{\underline{\omega}}f(x_{0}+\underline{\omega}r):=\frac{1}{2}\left(\frac{\partial}{\partial x_{0}}+\underline{\omega}\frac{\partial}{\partial r}\right)f_{\underline{\omega}}(x_{0}+\underline{\omega}r)=0,
\end{equation*}
on $\Omega_{D} \cap \mathbb{C}_{\underline{\omega}}$. We denote with $\mathcal{S}(\Omega_D)$ and $\mathcal{S}\mathcal{R}(\Omega_D)$ respectively the set of slice and slice regular functions on $\Omega_D$.


Furthermore, we introduce two notions related with slice functions, called the spherical value $f_{s}^{\circ}(x)$
and the  spherical derivative $f_{s}'(x)$ of $f$. They are defined as
\begin{equation} \label{qe1}
    \begin{split}
    			\left\{
    			\begin{array}{ll}
    		 f_{s}^{\circ}(x):=(f(x)+f(x^c))/2,\\
    				 f_{s}'(x):={\rm Im}(x)^{-1}(f(x)-f( x^c))/2,\\
    			\end{array}
    			\right.
    		\end{split}
\end{equation}
Note that $f^\circ_s$ and $f'_s$ are slice functions as well, induced respectively by $\alpha(z)$ and $\beta(z)/\operatorname{Im}(z)$. They are constant on every sphere $\mathbb{S}_x=\{x_0+\underline\omega|\underline x|:\underline\omega\in\mathbb{S}^{m-1}\}$, so they depend only on $\operatorname{Re}(x)$ and $|\operatorname{Im}(x)|$ and they decompose $f$ through the representation formula
\begin{equation*}
    f(x)=f^\circ_s(x)+\operatorname{Im}(x)f'_s(x).
\end{equation*}


In \cite{cgs}, Colombo, Gonz\'alez-Cervantes and Sabadini introduced a global non-constant coefficients differential operator to study slice regular functions. Here, we consider a slightly different operator,
  \begin{equation*}
     \overline{\theta}:=\partial_{x_0}+\frac{\underline x}{|\underline x|^2}\sum_{j=1}^m x_j\partial_{x_j},
 \end{equation*}
 defined on the class $\mathcal{C}^{1}(\Omega_{D}\backslash\mathbb{R})$ of $\mathbb{R}_{0,m}$-valued functions. The operator $\overline{\theta}$ is obtained by dividing the operator in \cite{cgs}  by $|\underline x|^2$.
 It was found that  slice regularity can be characterized as follows:
\begin{lemma} \cite{gp, per}\label{slice1} 
If $f\in \mathcal{C}^{1}(\Omega_{D})$ is a slice function, then $f$ is slice-regular if and only if $\overline{\theta}f=0$
on $\Omega_{D}\backslash\mathbb{R}$. If $\Omega_{D}\cap \mathbb{R}\neq \emptyset$ and $f\in \mathcal{C}^{1}(\Omega_{D})$ (not a priori a slice function), then $f$ is slice-regular if and only if $\overline{\theta}f=0.$

\end{lemma}

\section{Characterization of slice regular functions} \label{dsu1}
First, we give a  compact expression for the Dunkl spherical Dirac operator $\widetilde{\Gamma}_{\underline{\omega}}$ given in \eqref{ga1}. Denote \begin{equation} \label{eq1}
     \rho_{\underline\alpha}f(\underline{x}):=\frac{f(\underline{x})-f(\sigma_{\underline{\alpha}} \underline{x})}{\langle \underline{\alpha}, \underline{x}\rangle}=\int_{0}^{1}\partial_{\alpha}f(\underline{x}-t\langle \underline{\alpha},
     \underline{x}\rangle \underline{\alpha})\, {\rm d}t
\end{equation}
for $f\in \mathcal{C}^{1}(\mathbb{R}^{m})$, where $\partial_{\alpha}$ is the usual
directional derivative.
\begin{proposition}
\label{proposition expression Gamma} It holds that
    \begin{equation*}
       \gamma_{\kappa}+\Psi= -\underline{x}\sum_{\underline{\alpha}\in \mathscr{R}_+}\kappa(\underline{\alpha})\underline{\alpha}\,\,  \rho_{\underline\alpha},
    \end{equation*}
    where $\Psi$ is defined in \eqref{p21}.
    Thus, we have
    \begin{equation*}
        \widetilde{\Gamma}_{\underline\omega}=\Gamma_{\underline\omega}-\underline{x}\sum_{\underline{\alpha}\in \mathscr{R}_+}\kappa(\underline{\alpha})\underline{\alpha}  \, \rho_{\underline\alpha}.
    \end{equation*}
    \begin{proof}
        By definition \eqref{p21}, we have
        \begin{equation*}
            \Psi f(\underline x)=-\sum_{\underline{\alpha}\in \mathscr{R}_{+}}\kappa(\underline{\alpha})\underline{x}\wedge\underline{\alpha} \rho_{\underline\alpha} f(\underline x)-\sum_{\underline{\alpha}\in \mathscr{R}_+}\kappa(\underline{\alpha})f(\sigma_{\underline{\alpha}} \underline x).
        \end{equation*}
        It follows that
        \begin{equation*}
            \begin{split}
                ( \gamma_{\kappa}+\Psi)f(\underline x)&=\sum_{\underline{\alpha}\in \mathscr{R}_+}\kappa(\underline{\alpha})f(\underline x)-\sum_{\underline{\alpha}\in \mathscr{R}_+}\kappa(\underline{\alpha})\underline x\wedge\underline{\alpha} \rho_{\underline\alpha} f(\underline x)-\sum_{\underline{\alpha}\in \mathscr{R}_+}\kappa(\underline{\alpha})f(\sigma_{\underline{\alpha}} \underline x)\\
                &=\sum_{\underline{\alpha}\in \mathscr{R}_+}\kappa(\underline{\alpha})\left[f(\underline x)-f(\sigma_{\underline{\alpha}} \underline x)-\underline x\wedge\underline{\alpha} \rho_{\underline\alpha} f(\underline x)\right]\\
                &=\sum_{\underline{\alpha}\in \mathscr{R}_+}\kappa(\underline\alpha)\left[\langle \underline x, \underline{\alpha}\rangle \rho_{\underline\alpha} f(\underline{x})-\underline x\wedge\underline{\alpha} \rho_{\underline\alpha} f(\underline x)\right]\\
                &=-\underline{x}\sum_{\underline{\alpha}\in \mathscr{R}_+}\kappa(\underline{\alpha})\underline{\alpha}\,  \rho_{\underline\alpha} f(\underline x),
            \end{split}
        \end{equation*}
        where we have used \eqref{iw1} in the last step.
        Thus $\widetilde{\Gamma}_{\underline\omega}=\gamma_\kappa+\Gamma_{\underline\omega}+\Psi=\Gamma_{\underline\omega}- \underline{x}\sum_{\underline{\alpha}\in \mathscr{R}_+}\kappa(\underline{\alpha})\underline{\alpha} \rho_{\underline\alpha}$.
    \end{proof}
\end{proposition}

For later use, we give some properties of $\rho_{\underline\alpha}$. We will abuse the notations $\rho_{\underline\alpha}f(\underline{x})$  and $\rho_{\underline\alpha}f(x_{0}+\underline{x})$, which will not cause any confusion when $\mathscr{R}\subset \mathbb{R}^{m}$.
\begin{lemma} \label{rp1}
For any $\underline{\alpha}\in \mathscr{R}$,  the operator $\rho_{\underline\alpha}$ in \eqref{eq1} has the following properties:
    \begin{enumerate}
    \item $\displaystyle \rho_{\underline\alpha} f(|\underline{x}|)=0$;
   \item $\displaystyle
   \rho_{\underline\alpha}(\underline\omega)=\frac{2}{|\underline x|}\frac{\underline{\alpha}}{|\underline{\alpha}|^2};
 $
   \item $\displaystyle \rho_{\underline\alpha} (f)=\frac{2\underline{\alpha}}{|\underline{\alpha}|^2}f'_s$ for any slice function $f$;
   \item 
   $\rho_{\underline{\alpha}}(f\circ\sigma_{\underline{\alpha}})=-\rho_{\underline{\alpha}}f$;
    \item $\displaystyle \rho_{\underline\alpha}^2f=0$ for any function.
\end{enumerate}
\end{lemma}
\begin{proof}
    \begin{enumerate}
       \item It follows from $|\sigma_{\underline{\alpha}}(\underline x)|=|\underline x|$ that $\displaystyle \rho_{\underline\alpha} f(|\underline x|)=\frac{f(|\underline{x}|)-f(|\sigma_{\underline{\alpha}}\underline x|)}{\langle{\underline{\alpha}},\underline x\rangle}=0$.
        \item Since $\displaystyle \underline\omega(\sigma_{\underline{\alpha}}\underline x)=\frac{\sigma_{\underline{\alpha}}\underline x}{|\sigma_{\underline{\alpha}}\underline x|}=\frac{\sigma_{\underline{\alpha}}\underline x}{|\underline x|}=\underline\omega(\underline x)-\frac{2\langle{\underline{\alpha}},\underline{x}\rangle}{|\underline\alpha|^2|\underline x|}\underline{\alpha}$, we have
        \begin{equation*}
            \rho_{\underline\alpha}\underline\omega=\frac{2\langle {\underline{\alpha}}, \underline{x}\rangle}{|{\underline{\alpha}}|^2|\underline x|}\frac{{\underline{\alpha}}}{\langle {\underline{\alpha}}, \underline{x}\rangle }=\frac{2\underline{\alpha}}{|\underline{\alpha}|^2|\underline x|}.
        \end{equation*}
        \item \textcolor{black}{ For a slice function $f$,
        \begin{equation*}
            \rho_{\underline{\alpha}}(f)=\rho_{\underline{\alpha}}[f^\circ_s(x)+\operatorname{Im}(x)f'_s(x)]=\rho_{\underline{\alpha}}(f^\circ_s(x))+\rho_{\underline{\alpha}}(\operatorname{Im}(x)f'_s(x)).
        \end{equation*}
        The first term on the right side vanishes by (1). For the second term, by the facts $f_{s}'(\sigma_{\underline{\alpha}}x)=f_{s}'(x)$ and ${\rm Im}(x)=\underline{\omega}|\underline{x}|$, we have
        \begin{equation*}
            \rho_{\underline{\alpha}}(\operatorname{Im}(x)f'_s(x))=\rho_{\underline{\alpha}}(\underline{\omega})|\underline{x}|f_{s}'(x).
        \end{equation*}
        Now using (2), we obtain the formula (3).
        }
       \item 
        Since $\sigma_{\underline{\alpha}}^2=\operatorname{Id}$, for any $\underline{x}$ we have
        \begin{equation*}
            \rho_{\underline\alpha} (f\circ\sigma_{\underline{\alpha}})(\underline x)=\dfrac{f(\sigma_{\underline{\alpha}} \underline x)-f(\sigma_{\underline{\alpha}}^2 \underline x)}{\langle \underline{\alpha}, \underline x\rangle}=\dfrac{f(\sigma_{\underline{\alpha}}\underline x)-f(\underline x)}{\langle \underline{\alpha}, \underline x\rangle}=-\rho_{\underline\alpha} f(\underline x).
        \end{equation*}
        \item By (4), we have
        \begin{equation*}
            \rho_{\underline\alpha}^2 f\underline (\underline x)=\dfrac{\rho_{\underline\alpha} f(\underline x)-\rho_{\underline\alpha} f(\sigma_{\underline{\alpha}}\underline x)}{\langle \underline{\alpha}, \underline x\rangle}=\dfrac{\rho_{\underline\alpha} f(\underline x)-\rho_{\underline\alpha} f(\underline x)}{\langle \underline{\alpha},\underline x\rangle}=0.
        \end{equation*}
    \end{enumerate}
\end{proof}

\begin{proposition}[Characterization of sliceness]
\label{prop charcaterization sliceness}
    Let $f: \Omega_{D}\subset \mathbb{R}^{m+1}\to\mathbb{R}_{0, m}$ be a real analytic function. Suppose $\gamma_{\kappa}=(1-m)/2$ and $\kappa \not\in K^{{\rm sing}}$. Then the following conditions are equivalent:
    \begin{enumerate}
        \item  $f(x_{0}+\underline{x})$ is slice, (i.e., of the form \eqref{sf1});
        \item  $\displaystyle \widetilde{\Gamma}_{\underline\omega}f=0$ holds on $\Omega_{D}\backslash\mathbb{R}$;
    \end{enumerate}
    If one of the previous equivalent conditions holds, 
    then
    \begin{equation*}
        \Gamma_{\underline\omega}f= \gamma_{\kappa}\underline x\ \underline\alpha \rho_{\underline\alpha}f=-2 \gamma_{\kappa}\underline xf'_s,
    \end{equation*}
on $\Omega_{D}\backslash\mathbb{R}$ for every $\underline{\alpha}\in \mathscr{R}$. In particular,  $\underline{\alpha} \rho_{\underline{\alpha}}f=\underline{\beta} \rho_{\underline{\beta}}f$, for every $\underline{\alpha}, \underline{\beta}\in \mathscr{R}$ and slice function $f$.
\end{proposition}

\textcolor{black}{In order to prove the previous Proposition, we use a deformation of Fischer decomposition with respect to the Dunkl-Dirac operator $\underline{D}_{\kappa}$.
\begin{lemma}[\cite{oss}]\label{Dunkl-Dirac Fischer decomposition}
Let $\mathcal{P}_k$ denote the space of homogeneous polynomials of degree $k$ in $x_1,\dots x_m$ with values in $\mathbb{R}_{0,m}$, then $\mathcal{P}_k$ decomposes in
\begin{equation*}
\mathcal{P}_k=\bigoplus_{j=0}^{[k/2]}|\underline{x}|^{2j}\left(\mathcal{M}_{k-2j}+\underline{x}\mathcal{M}_{k-2j-1}\right),
\end{equation*}
where $\mathcal{M}_\ell$ denotes the space of homogeneous polynomials of degree $\ell$ which are in the kernel of $\underline{D}_{\kappa}$.
In other words, for any $P_k\in\mathcal{P}_k$ there exist unique $\mu_\ell,\nu_\ell\in\mathcal{M}_\ell$ such that
\begin{equation*}
P_k(x_1,\dots,x_m)=\sum_{j=0}^{[k/2]}|\underline{x}|^{2j}\left(\mu_{k-2j}(x_1\dots x_m)+\underline{x}\nu_{k-2j-1}(x_1,\dots x_m)\right).
\end{equation*}
\end{lemma}}

\begin{proof}[Proof of Proposition \ref{prop charcaterization sliceness}]
  $(1)\implies (2)$ Suppose $f$ is a slice function, i.e. $f(x_{0}+\underline{x})=\alpha(x_0,r)+\underline\omega\,\beta(x_0,r)$ with $r=|\underline{x}|$.  By the properties \eqref{pk12} of $\widetilde{\Gamma}_{\underline\omega}$, we have
    \begin{equation*}
    \begin{split}
\widetilde{\Gamma}_{\underline\omega}f(x_{0}+\underline{x})=&\,\widetilde{\Gamma}_{\underline\omega}\alpha(x_0,r)+\widetilde{\Gamma}_{\underline\omega}(\underline\omega\, \beta(x_0,r))\\=&\,(2 \gamma_{\kappa}+m-1)\underline\omega\,\beta(x_0,r)\\=&\,0.
    \end{split}
    \end{equation*}
$(2)\implies (1)$
For any fixed $x_{0}$, by Lemma \ref{Dunkl-Dirac Fischer decomposition} we  write $f(x_{0}+\underline{x})$ as
    \begin{equation} \label{dea}
        f(x_{0}+\underline x)=\sum_{k=0}^\infty P_{k}(x_0;\underline x)=\sum_{k=0}^\infty \left(\sum_{j=0}^{[k/2]}|\underline x|^{2j}(\mu^{(k)}_{k-2j}(x_0;\underline x)-\underline x \nu^{(k)}_{k-2j-1}(x_0;\underline x))\right),
    \end{equation}
    where for any fixed $x_0$, $\mu^{(k)}_\ell(x_0;\cdot),\nu^{(k)}_\ell(x_0;\cdot)$ are homogeneous Dunkl monogenic polynomial of degree $\ell$ on $\mathbb{R}^{m}$. 
  Furthermore, we remark that only the multiplicity function $\kappa\ge 0$ was considered in \cite{oss}, however it can be generalized for all non-singular $\kappa$, as done in the  scalar case \cite{ddo}. Similar to the classical case, we have  $\mu_{j}\perp \mu_{k}$ when $j\neq k$, and $\mu_{j}\perp \underline{x}\nu_{k}$ for any $j, k\in \mathbb{N}\cup\{0\}$ under the Fischer inner product. Now acting  $\widetilde{\Gamma}_{\underline\omega}$ term by term in the decomposition \eqref{dea}, together with the eigenvalues  of
    $\widetilde{\Gamma}_{\underline\omega}$ given in   \eqref{pk1} and \eqref{pk12},
     we see that
     \textcolor{black}{
     \begin{equation*}
    \begin{split}
    & \widetilde{\Gamma}_{\underline\omega}f(x_0+\underline{x})=\sum_{k=0}^{\infty}\sum_{j=0}^{[k/2]}|\underline{x}|^{2j}\left[\widetilde\Gamma_{\underline\omega}\left(\mu^{(k)}_{k-2j}(x_0;\underline x)\right)+\widetilde\Gamma_{\underline\omega}\left(\underline x\nu^{(k)}_{k-2j-1}(x_0;\underline x)\right)\right]\\
     &=\sum_{k=0}^{\infty}\sum_{j=0}^{[k/2]}|\underline{x}|^{2j}\left[(2j-k)\mu^{(k)}_{k-2j}(x_0;\underline x)+(2\gamma_\kappa+k-2j-1+m-1)\underline x\nu^{(k)}_{k-2j-1}(x_0;\underline x)\right]\\
     &=\sum_{k=0}^{\infty}\sum_{j=0}^{[k/2]}|\underline{x}|^{2j}\left[(2j-k)\mu^{(k)}_{k-2j}(x_0;\underline x)+(k-2j-1)\underline x\nu^{(k)}_{k-2j-1}(x_0;\underline x)\right].
     \end{split}
     \end{equation*}
     Thus,
    \begin{equation*}
       \widetilde{\Gamma}_{\underline\omega}f(x_{0}+\underline x)=0 \implies \left\{\begin{array}{ll}
       \mu^{(k)}_{k-2j}(x_0;\underline{x})=0&\text{ if }k\text{ is odd or if }k \text{ is even and } j\neq k/2\\
        \nu^{(k)}_{k-2j-1}(x_0;\underline{x})=0&\text{ if }k\text{ is even or if }k \text{ is odd and } j\neq \frac{k-1}{2}
       \end{array}
       \right.
    \end{equation*}
    In particular, we see that $f$ assumes the form
    \begin{equation*}
    f(x_0+\underline x)=\sum_{k=0}^\infty|\underline x|^{2k}\left(\mu_0^{(2k)}(x_0;\underline{x})+\underline{x}\nu_0^{2k+1}(x_0;\underline{x})\right).
    \end{equation*}
    But $\mu_0^{(2k)}(x_0;\underline x)=\mu_0^{(2k)}(x_0)$ and $\nu_0^{(2k+1)}(x_0;\underline x)=\nu_0^{(2k+1)}(x_0)$, thus by defining
    \begin{equation*}
    \alpha(x_0,|\underline{x}|)=\sum_{k=0}^\infty|\underline x|^{2k}\mu_0^{(2k)}(x_0)\qquad\beta(x_0,|\underline x|)=\sum_{k=0}^\infty|\underline x|^{2k+1}\nu_0^{(2k+1)}(x_0),
    \end{equation*}}
    we see that $f$ is of the form \eqref{sf1}, so $f$ is a slice function.
    
    For the last claim, the second equality follows from Lemma \ref{rp1} (3), while assuming $\widetilde{\Gamma}_{\underline\omega}f=0$ and by Proposition \ref{proposition expression Gamma} we have
    \begin{equation*}
        \Gamma_{\underline\omega}f=\underline x\sum_{\underline\alpha\in \mathscr{R}_+}\kappa(\underline\alpha)\underline\alpha \rho_{\underline\alpha}f=-2\underline x\sum_{\underline\alpha\in \mathscr{R}_+}\kappa(\underline\alpha)f'_s=-2 \gamma_{\kappa}\underline xf'_s= \gamma_{\kappa}\underline x\ \underline\alpha \rho_{\underline\alpha}f.
    \end{equation*}
\end{proof}
\begin{lemma} \label{st1}
Let $\gamma_{\kappa}=(1-m)/2$. Then $\overline{\theta}$ and $D_\kappa$ agree on slice functions of class $\mathcal{C}^{1}(\Omega_{D}\backslash\mathbb{R})$.
\end{lemma}
\begin{proof}
    If $f$ is a slice function, then from \cite[Proposition 3.2 (b)]{per} it holds
    \begin{equation*}
        \overline{\partial}_{x}f-\overline{\theta}f=2\gamma_{\kappa}f'_s.
    \end{equation*}
    Hence, we get
    \begin{equation*}
        \overline{\theta}f=\overline{\partial}_{x}f-2\gamma_{\kappa}f'_s=\overline{\partial}_{x}f+\sum_{\underline{\alpha}\in \mathscr{R}_+}\kappa(\underline{\alpha})\underline{\alpha} \rho_{\underline\alpha} f=D_\kappa f,
    \end{equation*}
    where we have used   $\underline{\alpha} \rho_{\underline\alpha} f=-2f'_s$ for any slice function $f$.
\end{proof}
\begin{theorem}[Characterization of slice regularity]  \label{sm1}  Assume $\gamma_{\kappa}=(1-m)/2$ and $\kappa \not\in K^{{\rm sing}}$. Let $f: \Omega_{D}\subset \mathbb{R}^{m+1}\to\mathbb{R}_{0, m}$ be a real analytic function.

{\rm (a)}  Suppose $f$ is slice, then $f$ is slice regular if and only if  $D_\kappa f=0$.

{\rm (b)} If  $\Omega_D\cap\mathbb{R}\neq\emptyset$,
then the following conditions are equivalent:
    \begin{enumerate}
        \item  $f$ is slice regular;
        \item  $f\in\ker D_\kappa \cap \ker \widetilde{\Gamma}_{\underline\omega}$.
    \end{enumerate}
\end{theorem}
\begin{proof}
   It follows from Proposition \ref{prop charcaterization sliceness}, Lemma \ref{st1} above and that $f$ is slice regular if and only if $\overline{\theta}f=0$ (see Lemma \ref{slice1}). 
\end{proof}
\begin{remark} \textcolor{black}{For certain specific root systems, the real analytic condition of $f$ may be relaxed to $\mathcal{C}^{1}(\Omega_{D})$. Here, we include this condition  to avoid delving into  too much technical  complexities associated with  negative  multiplicity functions.}
\end{remark}

Let $\Delta_{\mathbb{R}^{m+1}}=\sum_{j=0}^{m}\partial_{x_j}^{2}$ be the Laplace operator on $\mathbb{R}^{m+1}$. By choosing a specific root system in Theorem \ref{sm1}, we obtain  the following  connection between  slice regular functions and hyperbolic harmonic functions.
\begin{theorem}
 \label{sld1} Suppose $f$ is  slice regular, and denote $\underline{\omega}=\sum_{j=1}^{m}e_{j}\omega_{j}$ with $\omega_{j}\in \mathbb{R}$. Then  $\alpha(x_{0}, r) $ and   $\omega_{j} \beta(x_{0}, r)$ with $1\le j\le m-1$ are hyperbolic harmonic, i.e.
 \begin{equation} \label{a1}
    \left\{
  \begin{array}{ll}
    \displaystyle\left(x_{m}^{2}\Delta_{\mathbb{R}^{m+1}} -(m-1)x_{m} \partial_{x_{m}}\right)\alpha(x_{0}, r)=0, \\

  \displaystyle \left(x_{m}^{2}\Delta_{\mathbb{R}^{m+1}} -(m-1)x_{m} \partial_{x_{m}}\right) \omega_{j} \beta(x_{0}, r)=0,
  \end{array}
\right.
\end{equation}
while  $ \omega_{m}\beta(x_{0}, r)$ is the eigenfunction of the hyperbolic Laplace operator  corresponding to the eigenvalue $(1-m)$, i.e.
\begin{equation*}
 \left(x_{m}^{2}\Delta_{\mathbb{R}^{m+1}} -(m-1)x_{m} \partial_{x_{m}}\right) \omega_{m}\beta(x_{0}, r)=(1-m) \omega_{m}\beta(x_{0}, r).
\end{equation*}
\end{theorem}
\begin{proof} 

(1) Note that $\alpha(x_{0}, r)$ is harmonic in $x_{0}$ and $r$, i.e.  $(\partial^{2}_{x_{0}}+\partial_{r}^{2})\alpha(x_{0}, r)=0$. \textcolor{black}{Direct computation yields
the first equation in \eqref{a1}, i.e.
\begin{equation*}
\begin{split}
     &\left(x_{m}^{2}\Delta_{\mathbb{R}^{m+1}} -(m-1)x_{m} \partial_{x_{m}}\right)\alpha(x_{0}, r)\\
     = & \,\left(x_{m}^{2} \left(\partial_{x_{0}}^{2}+\partial_{r}^{2}+\frac{m-1}{r}\partial_{r}\right)-x_{m}^{2} \frac{(m-1)}{r}\partial_{r} \right)\alpha(x_{0}, r) \\
     = & \, x_{m}^{2}\left(\partial^{2}_{x_{0}}+\partial_{r}^{2}\right)\alpha(x_{0}, r)\\
     = & \, 0.
\end{split}
\end{equation*}}

(2) By Theorem \ref{sm1},  we have $D_{\kappa}f=0$.  This obviously implies that  $f$ is Dunkl-harmonic, i.e.
\textcolor{black}{\begin{equation*}
    \Delta_{\kappa}f=\overline{D_{\kappa}} D_{\kappa}f=0.
\end{equation*}}
We consider the following specific Dunkl-Cauchy-Riemann operator,
\begin{equation*}
    D_{\kappa}= \partial_{x_{0}}+e_{1}\partial_{x_{1}} +\cdots+e_{m-1}\partial_{x_{m-1}}+e_{m}T_{m},
\end{equation*}
where  \begin{equation*}\displaystyle T_{m}f(x)=\partial_{x_{m}}f(x)+\frac{1-m}{2}\frac{f(x)-f(\sigma_{e_{m}}x)}{x_{m}}.\end{equation*}
\textcolor{black}{ Straightforward  computation shows that \begin{equation*}
    		\begin{split}
    			T_{m}^{2}f=\left\{
    			\begin{array}{ll}
    				 \displaystyle  \partial_{x_{m}}^{2}f+\frac{1-m}{x_{m}}\partial_{x_{m}}f, & \mbox{when $f$ is even on $x_{m}$} ,\\
    				\displaystyle \partial_{x_{m}}^{2}f+\frac{1-m}{x_{m}}\partial_{x_{m}}f-\frac{1-m}{x_{m}^{2}}f, & \mbox {when $f$ is odd on $x_{m}$}.
    			\end{array}
    			\right.
    		\end{split}
\end{equation*}
Now since $\alpha(x_{0}, r)$ is even on $x_{m}$, the result in (1) is equivalent with \[\Delta_{\kappa}\,\alpha(x_{0}, r)=
       \left(\sum_{j=0}^{m-1}\partial_{x_{j}}^{2}+T_{m}^{2}\right) \,\alpha(x_{0}, r)=0.\]
Together with $\Delta_{\kappa}f=0$, we get
\begin{equation} \label{dbe1}\begin{split}
       &\Delta_{\kappa}\,\underline{\omega}\,\beta(x_{0}, r)=
       \left(\sum_{j=0}^{m-1}\partial_{x_{j}}^{2}+T_{m}^{2}\right) \,(f-\alpha(x_{0},r))=0.
\end{split}
\end{equation}}
The equation \eqref{dbe1} together with the fact  that $x_{j}\beta(x_{0}, r)$ is even on $x_{m}$ when $1\le j\le m-1$ and odd when $j=m$, yields
 the claimed relations for $\beta$, which concludes the proof.
\end{proof}
\begin{remark} The two equations of $\beta$ can also be verified by direct computation.
\end{remark}
\begin{remark} Note that  $\alpha(x_{0}, r)$ and $\beta(x_{0}, r)$ here can be $\mathbb{R}_{0,m}$-valued.
\end{remark}
\begin{remark} This result should be compared with the similar property of $k$-hypermonogenic functions (see e.g. \cite{el23,eo}).

\end{remark}
\begin{corollary}[Local decomposition]  Suppose $f$ is a slice regular function of the form \eqref{sf1}. Then the following holds when $x_{m}\neq 0$,
\begin{equation*}
    \begin{split}
        \alpha(x_{0}, r)=& \left(1- \frac{\underline{x}e_{m}}{m-1}\left(x_{m}\Delta_{\mathbb{R}^{m+1}} -(m-1)\partial_{x_{m}}\right)\right)f(x_{0}+\underline{x}),\\
        \beta(x_{0},r)=& \frac{e_{m}|\underline{x}|}{m-1}\left(x_{m}\Delta_{\mathbb{R}^{m+1}} -(m-1) \partial_{x_{m}}\right)f(x_{0}+\underline{x}).
    \end{split}
\end{equation*}
\end{corollary}
\begin{proof}
    It follows from Theorem \ref{sld1}.
\end{proof}
\begin{remark} Similarly when $x_{m}\neq 0$,  the  spherical derivative $f_{s}'(x_{0}+\underline{x})$ of $f$ can be expressed  locally as
\begin{equation*}
    f_{s}'(x_{0}+\underline{x})=\frac{e_{m}}{m-1} \left(x_{m}\Delta_{\mathbb{R}^{m+1}} -(m-1) \partial_{x_{m}}\right)f(x_{0}+\underline{x}),
\end{equation*}
instead of using \eqref{qe1} which uses conjugation in a global way.
\end{remark}
\section{Fueter's theorem revisited}\label{fuet1}
   Let $D$ be an open subset of $\mathbb{C}$, symmetric with respect to the real axis and let $F:D\to\mathbb{R}_{0,m}\otimes\mathbb{C}$ be a holomorphic stem function defined by
    \begin{equation*}
        F(z)=\alpha(s, t)+i\beta(s, t), \qquad z=s+it\in D.
    \end{equation*}
    Then  $F$
    induces a slice regular function
    \begin{equation} \label{isr2}
        f(x_{0}+\underline{x})=\alpha(x_{0}, |x|)+\frac{\underline{x}}{|\underline{x}|}\beta(x_{0}, |\underline{x}|)
    \end{equation}
    on the set  $ \Omega_{D}:=\{x_{0}+\underline{x}\in \mathbb{R}^{m+1}: (x_{0}, |x|)\in D\}$. When $m$ is odd,
    the classical Fueter's theorem (see e.g. \cite{fue, qian, sce}) states that  the function
    \begin{equation*}
        \Delta_{\mathbb{R}^{m+1}}^{\frac{m-1}{2}}   f(x_{0}+\underline{x})
    \end{equation*}
    is monogenic. It was already proved in Theorem \ref{sm1} that $f(x_{0}+\underline{x}) \in \ker D_\kappa$ with $\gamma_{\kappa}=(1-m)/2$.
    In this section, we show that the functions induced in Fueter's scheme, i.e.
    \begin{equation*}
        \Delta_{\mathbb{R}^{m+1}}  f(x_{0}+\underline{x}), \quad \Delta_{\mathbb{R}^{m+1}} ^{2}f(x_{0}+\underline{x}), \quad \ldots, \quad \Delta_{\mathbb{R}^{m+1}} ^{\frac{m-3}{2}}   f(x_{0}+\underline{x}),
    \end{equation*} are also in the kernel of certain $D_{\kappa}$ and there is an interesting shift property on the index $\gamma_{\kappa}$.  Note that most of the computations have already been done in \cite{fck}. Here we only provide a different viewpoint.

\begin{definition} \label{defi41} Let $f(x_{0}+\underline{x})$ be a slice  function of the form
\begin{equation}\label{sr21}
    f(x_{0}+\underline{x})=\alpha(x_{0}, r)+\underline{\omega}\, \beta(x_{0}, r)
\end{equation}
with $r=|\underline{x}|$.
We define
\begin{equation*}
\begin{split}
    \mathcal{V}_{n}f:=&\left(\frac{1}{r}\partial_{r}\right)^{n}\alpha(x_{0}, r)+\underline{\omega}\left(\partial_{r}\frac{1}{r}\right)^{n}\beta(x_{0}, r)\\
    :=& A(x_{0}, r)+\underline{\omega} B(x_{0}, r).
\end{split}
\end{equation*}
\end{definition}
\begin{lemma} \cite{fck} For a slice regular function $f(x)$ of the form \eqref{sr21}, we have
\begin{equation*}
    		\begin{split}
    			\mathcal{V}_{n}f(x_{0}+\underline{x})\sim \left\{
    			\begin{array}{ll}
    				  \Delta_{\mathbb{R}^{m+1}}^{n} f , & 0\le  n\le (m-1)/2\in \mathbb{N},\\
    				\Delta_{\kappa}^{n}f, & 0\le  n\le \gamma_{\kappa}+(m-1)/2\in \mathbb{N}.\\
    			\end{array}
    			\right.
    		\end{split}
\end{equation*}
Here the notation $\sim$ means that they are equal up to a non-zero real constant and $\Delta_{\kappa}=\sum_{j=0}^{m}T_{j}^{2}$ is the Dunkl Laplacian on $\mathbb{R}^{m+1}$.  (Note that there is an overlap for the domain of $n$ above. )
\end{lemma}
\begin{proof}
 \textcolor{black}{The second line is proved in  Lemma 4.2 in \cite{fck}, where the first line is proved implicitly there.}
\end{proof}

\begin{lemma} \cite{fck} \label{fg1} A slice function  $f$ on $\Omega_{D}\subset \mathbb{R}^{m+1}$ of the form  \eqref{sr21}
 is Dunkl monongenic, i.e., \,$D_{\kappa}f=0$,  if and only if $\alpha(x_{0}, r)$ and $\beta(x_{0}, r)$ satisfy the Vekua-type system
\begin{equation*}
    \left\{
  \begin{array}{ll}
    \partial_{x_{0}}\alpha-\partial_{r}\beta=\displaystyle \frac{2\gamma_{\kappa}+m-1}{r} \beta, \\
   \partial_{x_{0}}\beta+\partial_{r}\alpha=0.
  \end{array}
\right.
\end{equation*}
\end{lemma}

\begin{theorem} Suppose $f(x_{0}+\underline{x})$ is a slice regular function of the form \eqref{sr21}, then we have
\begin{equation*}
   \mathcal{V}_{n}f \in \ker D_{\kappa}, \qquad n\in \mathbb{N}\cup \{0\},
\end{equation*}
with $\gamma_{\kappa}=n+\frac{1-m}{2}$. In other words, for a fixed root system $\mathscr{R}$ with positive  multiplicity function $\kappa$ such that $\gamma_{\kappa}+\left(\frac{m-1}{2}\right)\in \mathbb{N}\cup \{0\}$, we have
\begin{equation} \label{qq1}
    		\begin{split}
    			\mathcal{V}_{n}f(x_{0}+\underline{x})\in \left\{
    			\begin{array}{ll}
    				  \ker \overline{\partial}_{x}, &  n=(m-1)/2\in \mathbb{N},\\
    				\ker D_{\kappa}, &  n=\gamma_{\kappa}+(m-1)/2\in \mathbb{N}.\\
    			\end{array}
    			\right.
    		\end{split}
    	\end{equation}
\end{theorem}
\begin{proof} By Lemma \ref{fg1}, it is sufficient to show that $A$ and $B$ (see Definition \ref{defi41}) satisfy the Vekua-type system. Indeed, we have
    \begin{equation*}
        \begin{split}
            \partial_{x_{0}}A-\partial_{r}B&=\left(\frac{1}{r}\partial_{r}\right)^{n}\{\partial_{x_{0}}\alpha(x_{0}, r)\}-\partial_{r}\left(\partial_{r}\frac{1}{r}\right)^{n}\beta(x_{0}, r)\\
            &=\left(\frac{1}{r}\partial_{r}\right)^{n}\{\partial_{r}\beta(x_{0}, r)\}-\partial_{r}\left(\partial_{r}\frac{1}{r}\right)^{n}\beta(x_{0}, r)\\
            &=\frac{2n}{r}\left(\partial_{r}\frac{1}{r}\right)^{n}\beta(x_{0}, r)\\
            &=\frac{2n}{r}B,
        \end{split}
    \end{equation*}
    \textcolor{black}{where the second equality follows from the fact that $\alpha$ and $\beta$ satisfy the classical Cauchy-Riemann equations,  and the third equality is by \cite[Lemma 4.1 (iv)]{fck}.}

    Similarly, we have
  \begin{equation*}
        \begin{split}
            \partial_{x_{0}}B+\partial_{r}A&=\left(\partial_{r}\frac{1}{r}\right)^{n} \{\partial_{x_{0}}\beta(x_{0},r)\}+
            \partial_{r}\left(\frac{1}{r}\partial_{r}\right)^{n}\{\alpha(x_{0}, r)\}\\
            &=\left(\partial_{r}\frac{1}{r}\right)^{n} \{\partial_{x_{0}}\beta(x_{0},r)\}+
            \left(\partial_{r}\frac{1}{r}\right)^{n}\{\partial_{r}\alpha(x_{0}, r)\}\\
            &=\left(\partial_{r}\frac{1}{r}\right)^{n}\{\partial_{x_{0}}\beta(x_{0}, r)+\partial_{r}\alpha(x_{0}, r) \}\\
            &=0,
        \end{split}
    \end{equation*}
   \textcolor{black}{where we have used the fact $ \partial_{r}\left(\frac{1}{r}\partial_{r}\right)^{n} =\left(\partial_{r}\frac{1}{r}\right)^{n}\partial_{r}$ in the second equality.}
  This completes the proof.
\end{proof}
\begin{remark} The first formula in \eqref{qq1} is indeed the classical Fueter's theorem, while the second one is  the  Fueter's theorem for Dunkl-monogenics obtained in \cite[Theorem 1.1]{fck}.
\end{remark}
As a summary,  the functions in  Fueter's scheme are in the kernel of the Dunkl-Cauchy-Riemann operator $D_{\kappa}$,  with the index $\gamma_{\kappa}$ given as follows,
\begin{equation*}
\begin{matrix} &f\longrightarrow & \Delta_{\mathbb{R}^{m+1}} f\longrightarrow & \cdots &\longrightarrow \Delta_{\mathbb{R}^{m+1}}^{\frac{m-1}{2}}f & \longrightarrow \mathcal{V}_{\frac{m+1}{2}}f \longrightarrow   & \mathcal{V}_{\frac{m+3}{2}}f \longrightarrow & \cdots\\
\\   \gamma_{\kappa}: &\frac{1-m}{2}\quad\quad & \frac{3-m}{2}\qquad  & \cdots & \qquad 0 & 1 &2& \cdots\\
\end{matrix}
\end{equation*}
This means that for the slice regular functions, the operator $\mathcal{V}_{n}$ (i.e. Laplacian $\Delta_{\mathbb{R}^{m+1}}$ or the Dunkl-Laplacian $\Delta_{\kappa}$) acts as a shift operator on the index $\gamma_\kappa$. Some closely related results are given recently in \cite{giulio}.

\section{New approach to generate monogenic functions by the inverse intertwining operator} \label{fna}

In this section, we always assume that the multiplicity function  $\kappa \not\in K^{{\rm sing}}$ in the Dunkl operator and $\gamma_{\kappa}=(1-m)/2$.
\textcolor{black}{Recall that the inverse intertwining operator (see e.g. \cite[Definition  6.5.1]{dx1}) is defined as follows:
\begin{equation*}
   \mathcal{T}f(x_{0}+\underline{x}):={\rm Exp} \left(\sum_{j=1}^{m} x_{j}T_{j}^{(\underline{y})}\right)f(x_{0}+\underline{y})\bigg|_{\underline{y}=0},
\end{equation*}
 where in the above formula, the operator $T_{j}^{(\underline{y})}$ acts  on the variable $\underline{y}$ of $f(x_{0}+\underline{y})$ first;  then we set set $\underline{y}=0$,   ultimately obtaining a function in $x_{0}+\underline{x}$.}







\begin{theorem} \label{sd} Suppose $f$ is a slice regular function of the form \eqref{sr21}, then the function $\mathcal{T}f(x_{0}+\underline{x})$
is monogenic, i.e. it is in the kernel of the generalized Cauchy-Riemann operator $\overline{\partial}_{x}=\partial_{x_{0}}+\sum_{j=1}^{m}e_{j}\partial_{x_{j}}$.
\end{theorem}
\begin{proof} It is sufficient to verify
\begin{equation*}
\begin{split}
     \left(\partial_{x_{0}}+\sum_{j=1}^{m}e_{j}\partial_{x_{j}}\right) \mathcal{T}f(x_{0}+\underline{x})&= \mathcal{T}\left( \left(\partial_{x_{0}}+\sum_{j=1}^{m}e_{j}T_{j}\right) f(x_{0}+\underline{x})\right)\\
     &=\mathcal{T} 0=0.
\end{split}
\end{equation*}
The first step is by the property of the inverse intertwining operator
\begin{equation*}
    \partial_{x_j}\circ \mathcal{T}=\mathcal{T} \circ T_{j}
\end{equation*}
at least formally, see \cite[Proposition 6.5.2]{dx1}.
The second step is by Theorem  \ref{sm1}.
\end{proof}
Some remarks are now in the order.
\begin{remark} As opposed to the classical Fueter's theorem (see e.g. \cite{ qian, sce}), there is no significant difference between even and odd dimensions in this approach. However, singular parameters $\kappa$ should be excluded.
\end{remark}
\begin{remark} Our formula looks similar with the other well-known approach to construct monogenic functions, namely the
Cauchy-Kovalevskaya extension \textcolor{black}{(see e.g. \cite[Theorem 2]{deso})}, given by
\begin{equation*}
    f(x_{0}+\underline{x})={\rm Exp}\left(-x_{0} \sum_{j=1}^{m}e_{j}\partial_{x_j}\right)f(\underline{x}).
\end{equation*}
\end{remark}
\begin{remark} The formal inverse intertwining operator $\mathcal{T}$ exists for any multiplicity function, see \cite{dx1}.
On the other hand, it is known that the Dunkl's intertwining operator for the rank 1 root system can be expressed explicitly using the Weyl fractional integral, see e.g.\,\cite{dl, dx1}. It will be interesting to write $\mathcal{T}$ in terms of classical fractional integrals and ordinary derivatives for the rank 1 root system.
Then one can give more precise convergence discussions.
\end{remark}
\begin{remark} It is interesting to ask the domain of $ \mathcal{T}f(x_{0}+\underline{x})$. From the proof of Theorem \ref{sd}, the result holds when the inverse intertwining operator is well defined and the  inverse intertwining relations hold.

\end{remark}
At the end of this section, we give some examples mainly related with the Dunkl operators on the line.
\begin{ex} For the polynomial $x_{0}+\underline{x}=x_{0}+e_{1}x_{1}+e_{2}x_{2}+e_{3}x_{3}$, the classical Fueter's theorem in \cite{sce} produces a trivial monogenic function.
 If we choose $D_{\kappa}=\partial_{x_{0}}+e_{1}\partial_{x_{1}}+e_{2}\partial_{x_{1}}+e_{3}T_{3}$, where
\begin{equation}\label{pq1}
    T_{3}f(x)=\partial_{x_{3}}f(x)-\frac{f(x)-f(\sigma_{e_{3}}x)}{x_{3}},
\end{equation} then we have
 \begin{equation*}
     \mathcal{T}(x_{0}+\underline{x})=x_{0}+e_{1}x_{1}+e_{2}x_{2}-e_{3}x_{3},
 \end{equation*}
 which is a homogeneous monogenic polynomial of degree 1.
Similarly, if we put the Dunkl operator in the variable $x_{2}$, it yields
\begin{equation*}
      \mathcal{T}(x_{0}+\underline{x})=x_{0}+e_{1}x_{1}-e_{2}x_{2}+e_{3}x_{3}.
 \end{equation*}
\end{ex}
  \begin{ex} For the polynomial \begin{equation*}
     x^{2}=(x_{0}+\underline{x})^{2}=(x_{0}+e_{1}x_{1}+e_{2}x_{2}+e_{3}x_{3})^{2},
  \end{equation*} Fueter's theorem in \cite{sce} again produces a constant, which is clearly monogenic. While choosing as before $D_{\kappa}=\partial_{x_{0}}+e_{1}\partial_{x_{1}}+e_{2}\partial_{x_{1}}+e_{3}T_{3}$, where $T_{3}$ is the operator defined in \eqref{pq1},  we obtain
  \begin{equation*}
       \mathcal{T}(x_{0}+\underline{x})^2=x_0^2-x_1^2-x_2^2+x_3^2+2x_0(e_1x_1+e_2x_2-e_3x_3),
  \end{equation*}
      which is a homogeneous monogenic of degree 2.

      Furthermore, if we  consider the root system $\mathscr{R}=\{\pm e_j\}_{j=1}^3$, with $D_{h}=\partial_{x_{0}}+e_{1}T_{1}+e_{2}T_2+e_{3}T_{3}$, in which
 \begin{equation*}
     T_{j}f(x)=\partial_{x_{j}}f(x)+\kappa(e_j)\frac{f(x)-f(\sigma_{e_j}x)}{x_{j}},  \qquad i=1,2,3,
 \end{equation*} our approach  produces a family of monogenic functions under the condition $\gamma_\kappa=\sum_{i=1}^3\kappa (e_i)=-1$:
 \begin{equation*}
      \mathcal{T}(x_{0}+\underline{x})^2=x_0^2-\sum_{i=1}^3x_i^2(1+2\kappa(e_i))+2x_0\left(\sum_{i=1}^3e_ix_i(1+2\kappa (e_i))\right).
 \end{equation*}
 Indeed, we have
 \begin{equation*}
\overline{\partial}_{x} \mathcal{T}(x_{0}+\underline{x})^2=-4x_0\left(1+\sum_{i=1}^3\kappa(e_i)\right)=0.
 \end{equation*}

 This can also be extended to the general Clifford setting.  Consider the polynomial on $\mathbb{R}^{m+1}$, \begin{equation*}
     x^{2}=(x_{0}+\underline{x})^2=x_0^2-\sum_{i=1}^mx_i^2+2x_0\sum_{i=1}^m e_ix_i
 \end{equation*} and the root system $\mathscr{R}=\{\pm e_i\}_{i=1}^m$. Then we have
 \begin{equation*}
     \mathcal{T}x^{2}=x_0^2-\sum_{i=1}^m x_i^2(1+2\kappa(e_i))+2x_0\left(\sum_{i=1}^me_ix_i(1+2\kappa(e_i))\right).
 \end{equation*}
 These are monogenic functions under the condition $\gamma_\kappa=\sum_{i=1}^{m}\kappa(e_i)=(1-m)/2$. Indeed, straightforward computation shows that
 \begin{equation*}
     \overline{\partial}_{x}\mathcal{T}x^{2}=4x_0\left(\frac{1-m}{2}-\sum_{i=1}^m\kappa(e_i)\right)=0.
 \end{equation*}
  \end{ex}
  \section{Concluding remarks}
This paper establishes the groundwork for a research project to investigate the interactions among Dunkl, Clifford and slice analysis. As shown, the inverse intertwining operator in Dunkl theory, serves as a potent tool that naturally generates examples of monogenic functions, in a different way than Fueter's celebrated theorem. These elucidated connections are poised to yield novel insights into slice regular functions, leveraging methodologies  from Dunkl theory and hyperbolic function theory. Noteworthy outcomes will include the derivation of novel mean value properties and the Cauchy integral formula, among others.  Moreover, the reflections can be shown to be involutional anti-automorphisms which is the basis of the construction of slice functions in the context of alternating $^\ast$-algebras in \cite{ghiper}. Currently, only a small part of these tools has been taken into consideration and in the future we plan to analyze them in more depth.

   \section*{Acknowledgements} This paper has its origin in a stay of the first  author at the Clifford Research Group of Ghent University, Belgium in 2023. He would like to thank the hospitality there. The first author was  partially supported by GNSAGA of INdAM through the project ``Teoria delle funzioni ipercomplesse e applicazioni". The third author was supported  by NSFC with Grant No.12101451 and  China Scholarship Council. \textcolor{black}{We appreciate the referees'  valued comments,  which improved the clarity of the paper.}



\bibliographystyle{amsplain}

\end{document}